\documentclass[10pt,preprint]{elsarticle}
\usepackage{lineno,hyperref,amsfonts,amsmath,amsthm,bigints}
\usepackage[a4paper, total={6in, 8in}]{geometry}
\usepackage[usenames, dvipsnames]{color}
\modulolinenumbers[3]

\journal{Journal of \LaTeX\ Templates}

\newcommand{\delp}{\Delta_p}      
\newcommand{\Om}{\Omega}
\newcommand{\del}{\delta}
\newcommand{\lam}{\lambda}
\newcommand{\alp}{\alpha}
\newcommand{\pa}{\partial}
\newcommand{\rtar}{\rightarrow}
\newcommand{\ga}{\gamma}
\newcommand{\Sig}{\Sigma}
\newcommand{\eps}{\epsilon}

\newtheorem{thm}{Theorem}[section]
\newtheorem{lem}{Lemma}[section]
\newdefinition{rmk}{Remark}[section]
\newproof{pf}{Proof}

\begin{document}

\begin{frontmatter}

\title{Multiplicity results for a quasilinear equation with singular nonlinearity\tnoteref{mytitlenote}}
\tnotetext[mytitlenote]{The first author is supported by DST-Inspire Faculty Award MA-2013029.\\
The second author is supported by NBHM Fellowship No: 2/39(2)/2014/NBHM/R$\&$D-II/8020/June 26,2014}

\author{Kaushik Bal, Prashanta Garain \fnref{myfootnote}}
\address{Dept of Mathematics and Statistics\\
 Indian Institute of Technology\\
 Kanpur-208016\\
 Uttar Pradesh, India}
 
\cortext[mycorrespondingauthor]{Corresponding Author}
\ead{kaushik@iitk.ac.in,pgarain@iitk.ac.in}

\begin{abstract}
For an open, bounded domain $\Om$ in $\mathbb{R}^N$ which is strictly convex with $C^2$ boundary, we show that there exists a $\land>0$ such that the singular quasilinear problem
\begin{eqnarray*}
&-\delp u =\cfrac{\lambda}{u^{\del}}+u^q\,\,\mbox{in}\,\,\Om\\
&u=0\,\,\mbox{on}\,\,\partial\Om;\, \,\,u>0\,\,\mbox{in}\,\,\Om
\end{eqnarray*}
admits atleast two solution $ u$ and $v$ in $W^{1,p}_{loc}(\Om)\cap L^{\infty}(\Om)$ for any $\del>0$ and $0<\lam<\land$ provided $1<p<N$ and $p-1<q<\frac{p(N-1)}{N-p}-1$.\\
Moreover the solutions $u$ and $v$ are such that $u^{\alp}$ and $v^{\alp}$ are in $W^{1,p}_0(\Om)$ for some $\alp>0$.
\end{abstract}


\end{frontmatter}

\linenumbers


\section{Introduction}
The aim of this work is to establish the multiplicity of the solution to the quasilinear elliptic problem given by:
\begin{equation}\label{mp}
\begin{array}{lc}
&-\delp u =\cfrac{\lambda}{u^{\del}}+u^q\,\,\mbox{in}\,\,\Om\\
&u=0\,\,\mbox{on}\,\,\partial\Om;\, \,\,u>0\,\,\mbox{in}\,\,\Om
\end{array}
\end{equation}
where $\Om$ is a open, bounded domain which is strictly convex with $C^2$ boundary.  We also assume that either $1<p<N$ and $p-1<q<p_{*}-1$ where $p_{*}:=\frac{p(N-1)}{N-p}$ is the Serrin Exponent. Also $\del>0$ and $\lam>0$ is assumed. We show that under the given conditions the problem (\ref{mp})
admits atleast two generalized solution $u,v\in W^{1,p}_{loc}(\Omega)$.

 By generalized solution we mean a function $u\in W^{1,p}_{loc}(\Om)\cap L^{\infty}(\Omega)$ such that $\frac{\phi}{u^{\delta}}\in L^1(\Om)$ and satisfying \[\int_{\Om}|\nabla u|^{p-2}\nabla u\cdot\nabla \phi \;dx=\lambda\int_{\Om}\frac{\phi}{u^{\delta}} dx+\int_{\Om}u^q\phi \;dx\] for every $\phi\in W^{1,p}_0(\omega)$ with $\omega\subset\subset\Om$. The Dirichlet boundary condition will be interpreted as that a suitable positive power of $u$ belongs to $W^{1,p}_0(\Om)$.

Problems with singular nonlinearity has been studied extensively in the literature. Crandall et al \cite{CrRaTa} in a famous work considered the problem 
$$-\Delta u=f(u)\;\;\mbox{in}\;\;\Omega;\;\;u=0\;\;\mbox{in}\;\;\partial\Omega$$
with $f$ singular near $0$ and showed that there exists a unique classical solution $u\in C^2(\Omega)\cap C(\bar\Omega)$. In particular if one considers $f(x)=\frac{1}{x^{\delta}}$ then the classical solution exists for any $\delta>0$. Later Lazer-Mckenna \cite{LaMc} showed that the unique classical solution $u$ is also in $H^1_0(\Omega)$ iff $0<\delta<3$. They also showed that the solution belongs to $C^1(\bar\Omega)$ if $0<\delta<1$. 

Haitao \cite{Ha} studied the perturbed singular problem
\begin{eqnarray}\label{ha}
&-\Delta u=\cfrac{\lambda}{u^{\del}}+u^p\,\,\mbox{in}\,\,\Om\\
&u=0\,\,\mbox{on}\,\,\partial\Om;\, \,\,u>0\,\,\mbox{in}\,\,\Om\nonumber
\end{eqnarray}
and showed that there exists two weak solutions for $\lambda<\Lambda$, no solution for $\lambda>\Lambda$ and atleast one solution for $0<\delta<1<p\leq \frac{N+2}{N-2}$ and some $\Lambda>0$.
This paper was generalized for p-Laplacian by Giacomoni et at \cite{GiScTa} who showed the existence of two solution for $0<\delta<1$ and $p-1<q\leq p^{*}-1$. The restriction of $\delta$ was removed in the paper by Boccardo and Orsina \cite{BoOr} who studied the problem
$$-div (M(x)\nabla u)=\frac{f(x)}{u^{\delta}}\;\;\mbox{in}\;\;\Omega;\;\;u=0\;\;\mbox{in}\;\;\partial\Omega$$ and showed the existence of a $u\in W_0^{1,1}(\Omega)$ such that for all $\omega\subset\subset\Omega$ there exists $c_{\omega}$ with $u\geq c_{\omega}>0$ in $\omega$ for any $\delta$. Recently the problem has been generalized by Canino et al \cite{CaScTr} for the p-Laplacian. Some results regarding the corresponding parabolic problems can also be found in Badra et al \cite{BaBaGi1},\cite{BaBaGi2} and the reference therein. Arcoya-Merida \cite{ArMo} studied the perturbed problem (\ref{ha}) and showed the multiplicity of solution in $H^1_{loc}(\Om)\cap L^{\infty}(\Om)$ for any $\delta >0$, where solution is meant in a generalized sense as presented in our case. 

Encouraged by the work of Arcoya-Merida \cite{ArMo} we generalize their result for the p-laplacian, but for the degeneracy of the p-laplacian the same proof will not work in our case. We started by the standard approach of studying the multiplicity of the regularized problem. To this aim Kelvin transform was used to obtain the boundary estimates on the solutions of the regularized problem in \cite{ArMo} which fails for p-Laplacian see Lindqvist \cite{Lind}. Moreover application of the moving plane method is also a problem due to the degeneracy of the p-laplacian at the critical points. We overcome this difficulty by proving an uniform Hopf Lemma by modifying the arguments of Vazquez \cite{Vaz} (also see Peral \cite{Pe}) in combination with a delicate application of Moving Plane technique by combining some of our ideas with that of Castorina-Sanch\'on \cite{CaSa} to arrive at the required estimate, what this did is gave us a uniform neighbourhood of the boundary for any solution to allow the blow-up analysis of Gidas-Spruck \cite{GiSp}, which required segrerating the maximas of $u_n$ in some interior of the boundary independent of $n$.  The existence was obtained by using a bifurcation result of Ambrosetti-Arcoya \cite{AmAr}, and then we pass to the limit to obtain our desired result.
It should be noted that the same problem has been handled in the paper by Giacomoni et al \cite{GiScTa} where they have proved the existence of atleast two solution in $W^{1,p}_0(\Om)$ provided $0<\delta<1$ among other things. So in this paper we will mainly concentrate on the case $\delta\geq 1$.\\
Before we move to our main result let us define the set 
$$\mathbb{E}=\{(p,q):1<p<N,\;p-1<q<p_{*}-1\; \}$$
where $p_{*}$ is the Serrin exponent given by $p_{*}=\frac{p(N-1)}{N-p}$. For the rest of the paper
we will assume $(p,q)\in \textbf{E}$ and $\Om$ to be open bounded domain in $\mathbb{R^N}$ with $C^2$ boundary unless otherwise mentioned.

\section{Main Result}
\begin{thm}\label{MT}
Given $\delta>0$ there exist $\Lambda>0$  such that the problem (\ref{mp}) admits atleast two solution $u,v\in W^{1,p}_{loc}(\Omega)$ provided $\Om$ is strictly convex with $(p,q)\in\mathbb{E}$ and for $0<\lambda<\Lambda$. Moreover there exists $\alpha>0$ such that $u^{\alpha}, v^{\alpha}\in W^{1,p}_0(\Omega)$. 
\end{thm}

\begin{rmk}
Note that for $p=2$ the above result only proves the multiplicity result for the range   $1<q<2_{*}-1$ which is less than $1<q<2^{*}-1$. But this is due to the Liouville Theorem for the p-Laplacian. For $p=2$ the method used here gives the multiplicity result in the full range as in Arcoya-Merida \cite{ArMo}.
\end{rmk}

\section{Preliminaries}
Before we begin with the proof of our main result we start by proving a few lemmas:
\begin{lem}\label{pursing}
Given $\del>0$ the problem 
\begin{equation}\label{pursinre}
-\delp u=\frac{\lam}{(u+\frac{1}{n})^{\del}}\,\mbox{in}\;\Om;\;\;u=0\,\,\mbox{on}\,\,\partial\Om
\end{equation}
 admits a unique positive solution in $W^{1,p}_0(\Omega)\cap L^{\infty}(\Omega)$ for each $n\in\mathbb{N}$. Moreover, $u_n$ is increasing w.r.t $n$ and $u_n(x)>c_{\omega}>0$ for all $\omega\subset\subset\Om$ and $c_{\omega}$ depends only on $\omega$ and not on $n$. Also, $||u_n||_{\infty}\leq M {\lambda}^{\frac{1}{\delta+p-1}}$ for all  $n\in\mathbb{N}$ with $M$ independent of $n$.
\end{lem}

\begin{proof}[\textbf{Proof of Lemma \ref{pursing}}]
Fix $v\in L^p(\Omega)$ and $n\in\mathbb{N}$, consider $J_{\lambda}:W^{1,p}_0(\Om)\to \mathbb{R}$ defined as $$J_{\lambda}(u)=\frac{1}{p} \int_{\Om} |\nabla u|^p dx-\lambda \int_{\Om} \frac{u}{(|v|+\frac{1}{n})^{\delta}} dx$$
Clearly, $J_{\lambda}$ is continuous, coercive and strictly convex in $W^{1,p}_0(\Om).$ Hence there exists a unique minimizer $w\in W^{1,p}_0(\Om)$. Define $S:L^{p}(\Omega)\rightarrow L^p(\Omega)$ by
$$S(v)=(-\Delta_{p})^{-1}(\frac{\lambda}{(\vert{v}\vert+\frac{1}{n})^{\delta}}):=w$$
Taking $w$ as a test function in the equation $-\Delta_{p}{w}=\frac{\lambda}{(\vert{v}\vert+\frac{1}{n})^{\delta}}$, we get
\[\int_{\Omega}{\vert{\nabla{w}}\vert}^{p} dx<{\lambda}n^{\delta}\int_{\Omega}\vert{w}\vert dx\]
By the Poincar\'e Inequality,  $||w||_p\leq C$ where $C$ is independent of $w$ but depends on $n$ and $\lambda.$\\
Again since $S$ is continuous and compact in $L^p(\Om)$ we have by Schauder fixed point theorem, the existence of a fixed point say $w$. 
 Hence by strong maximum principle (Theorem \ref{scp2}) we have $w>0$ in $\Omega$ satisfying, 
 \begin{equation*}
 \delp w=\frac{\lambda}{(w+\frac{1}{n})^{\delta}};\;\; w\in W^{1,p}_0(\Om)
 \end{equation*}
uniqueness is a simple consequence of the monotonicity of the singularity.
 
Denote $u_i$ to be the solution of the equation
\begin{equation}\label{p_i}
-\delp u=\lambda\big(u+\frac{1}{i}\big)^{-\del}\,\mbox{in}\;\Om;\;\;u=0\,\,\mbox{on}\,\,\partial\Omega
\end{equation} 
for $i=1,2,..$

Subtracting equation (\ref{p_i}) for $i=n$ from $i=n+1$ and multiplying with $(u_n-u_{n+1})^{+}$ we have,
\begin{multline}\label{mono}
\int\limits_{\Om} (|\nabla u_n|^{p-2}\nabla u_n-|\nabla u_{n+1}|^{p-2}\nabla u_{n+1})\cdot\nabla (u_n-u_{n+1})^{+}dx\\
\leq \lambda\int\limits_{\Om}\big[(u_n+\frac{1}{n+1})^{-\delta}- (u_{n+1}+\frac{1}{n+1})^{-\delta}\big] (u_n-u_{n+1})^{+} dx
\end{multline}

From the Algebraic Inequality we get for $p\geq 2$,  $$\int\limits_{\Om}|\nabla u_n|^{p-2}\nabla u_n-|\nabla u_{n+1}|^{p-2}\nabla u_{n+1},\nabla (u_n-u_{n+1})^{+} dx\geq C_p ||\nabla (u_n-u_{n+1})^{+}||^p\geq 0$$ 
Again for $1<p<2$, we have $$\int\limits_{\Om} |\nabla u_n|^{p-2}\nabla u_n-|\nabla u_{n+1}|^{p-2}\nabla u_{n+1},\nabla (u_n-u_{n+1})^{+} dx\geq C_p \frac{||u_n-u_{n+1}||^2}{(||u_n||+||u_{n+1}||)^{2-p}}\geq 0$$
Again from the monotonicity of $f(x)=x^{-\delta}$ we have,
\begin{equation*}
\int\limits_{\Om}[(u_n+\frac{1}{n+1})^{-\delta}- (u_{n+1}+\frac{1}{n+1})^{-\delta}] (u_n-u_{n+1})^{+} dx\leq 0
\end{equation*} 
Combining this with (\ref{mono}) we have,
 \[||(u_{n}-u_{n+1})^{+}||=0\] which combining with the boundary conditions gives $$(u_{n}-u_{n+1})^{+}=0$$ 
therefore $u_n$ is monotonically increasing w.r.t $n$.

By Strong Maximum principle of Vazquez \cite{Vaz}, $u_1>0$ in $\Om$ where $u_1$ solves the equation $$-\delp u=\frac{\lam}{(u+1)^{\del}}\,\mbox{in}\;\Om;\;\;u=0\;\;\mbox{in}\;\;\partial\Om$$ Hence, using regularity theorem of Lieberman \cite{Lie} one can conclude that $u_n\in C^1(\bar{\Om})$ for all $n\in\mathbb{N}$. Therefore from monotonicity of solutions we can conclude that $u_n>u_1$ in $\Om$ and hence $$u_n>c_{\omega}>0 \;\;\mbox{for}\;\;\omega\subset\subset\Omega$$ with $c_{\omega}$ is independent of $n.$\\ Now to show the uniform boundedness of the solutions we assume, $v=u_{n}$ and let $\lambda=1$.
For $k\geq 1$, choose $$\phi:=G_{k}(v)=\begin{cases}
v-k\;\mbox{if}\; v>k\\
0\;\;\;\mbox{if}\;\;\; v\leq k
\end{cases}$$
and define, $A(k)=\{x\in\Omega:v>k\}$.
So for $1<k<h$ we have, $A(h)\subset A(k)$.\\
Since, $-\Delta_{p}{v}=\frac{1}{(v+\frac{1}{n})^{\delta}}<\frac{1}{v^{\delta}}$ hence
\begin{equation}{\nonumber}
\int_{A(k)}|\nabla v|^p dx<\int_{A(k)}\frac{v-k}{v^{\delta}}dx\leq{|A(k)|^{\frac{1}{p'}}}
{\vert\vert{(v-k)}\vert\vert}_{L^{p}(A(k))}<c{\vert{A}(k)\vert}^{\frac{1}{p^{'}}}{\vert\vert{\nabla{v}}\vert\vert}_{L^{p}(A(k))}
\end{equation}
By the Poincar\'e Inequality and Sobolev embedding theorem we have,
\begin{equation}{\nonumber}
{\vert\vert{v}\vert\vert}^{p-1}_{L^{p*}(A(k))}<\frac{c}{S^{p-1}}{\vert{A}(k)\vert}^{\frac{1}{p^{'}}}.
\end{equation}
where $c$ and $S$ are the Poincare and Sobolev constant respectively with $p'=\frac{p}{p-1}$.
Using the above inequalities we get,
\begin{equation}{\nonumber}
{\vert{A(h)}\vert}\leq\big({\frac{c}{S^{p-1}}}\big)^{\frac{p^{*}}{p-1}}\frac{1}{(h-k)^{p^{*}}}{\vert{A(k)\vert}^{\frac{p^{*}}{p}}}.\\
\end{equation}
If we choose $d=({\frac{c}{S^{p-1}}})^{\frac{p^{*}}{p-1}},\alpha=p^{*},\beta=\frac{p^{*}}{p}>1.$\\
Hence for $h>k>1,$
\begin{equation}{\nonumber}
{\vert{A(T)}\vert}=0.
\end{equation}
which implies $v\in L^{\infty}(\Omega)$ and $||v||_\infty\leq T$ for some $T$ independent of $n.$\\
Now, for any $\lambda>{0}$ suppose $v$ satisfies
\begin{equation}{\nonumber}
\int_{\Omega}{\vert\nabla{v}\vert}^{p-2}\nabla{v}.\nabla{\phi}<{\lambda}\int_{\Omega}\frac{\phi}{v^{\delta}}
\end{equation}
for all $\phi\in [W_0^{1,p}(\Omega)]^{+}$.\\
Choosing $w=(\frac{1}{\lambda})^{\frac{1}{\delta+p-1}}v$ we see that $w$ satisfies:
\begin{equation}{\nonumber}
\int_{\Omega}{\vert\nabla{w}\vert}^{p-2}\nabla{w}.\nabla{\phi}<\int_{\Omega}\frac{\phi}{w^{\delta}},\;\forall\;\phi\in{W}_{0}^{1,p}({\Omega}),\;\phi>0.
\end{equation}
Hence from the case $\lambda =1$ we have,
\begin{equation}{\nonumber}
\begin{aligned}
{\vert\vert{w}\vert\vert}_{\infty}\leq{T}\;\mbox{which implies}\;{\vert\vert{v}\vert\vert}_{\infty}\leq{T}{\lambda}^{\frac{1}{\delta+p-1}}
\end{aligned}
\end{equation}
\end{proof}

\begin{lem}\label{wosing} 
There exists $\delta_0>0$ such that every bounded non-trivial solution $u$ of the problem $-\Delta_p u=u^q\;\mbox{in}\;\Om$ satisfies $||u||_{\infty}>\delta_0$.
\end{lem}
\begin{proof}[\textbf{Proof of Lemma \ref{wosing}}]
Assume there exists a sequence $u_n$ of non-trivial solutions such that
$||u_n||_{\infty}\to 0$ as $n\to\infty$.
Define $v_{n}(x):=u_n(x)||u_n||_{\infty}^{-1}$ then $||v_n||_{\infty}=1$.\\
 Since $u_n$ satisfies $-\delp u=u^q$ hence we have, 
 \begin{equation}\nonumber
 \Delta_{p}{v_{n}}={\vert\vert{u}_{n}\vert\vert}_{\infty}^{q-p+1}{v}_{n}^{q}:=f_{n}
 \end{equation}
Since $f_n$ are uniformly bounded in $L^{\infty}(\Om)$ for sufficiently large $n$, we have by Tolksdorf regularity results \cite{To} that
$||v_n||_{C^{1,\beta}(\overline{\Omega})}\leq M$ for some $\beta\in(0,1)$ and $M$ independent of $n$. By Ascoli-Arzela upto a subsequence $v_n\to v$ in $C_0^1(\overline{\Omega})$, but that would imply $v=0$, thanks to Lemma 1.1 of Azizieh-Clement \cite{AzCl} contradicting that $||v_n||_{\infty}=1$.
\end{proof}

\begin{lem}\label{nonex}
There exists $\bar{\Lambda}>0$ such that for all $\lambda\geq \bar{\Lambda}$ the problem 
\begin{equation}\label{dop}
\begin{array}{lc}
-\delp u=\cfrac{\lam}{(u+\frac{1}{n})^{\del}}+u^q\,\mbox{in}\;\Om;\\
u=0\,\,\mbox{on}\,\,\partial\Om;\;u>0\;\mbox{in}\;\Om
\end{array}
\end{equation}
does not admit any weak solution $u\in W^{1,p}_0(\Omega)$. 
\end{lem}

\begin{proof}[\textbf{Proof of Lemma \ref{nonex}}]
Let us assume $\phi_1\geq 0$ to be the first eigenfunction corresponding to the first eigenvalue $\lambda_1$ of the operator $-\delp$ i.e,
\begin{equation*}
-\Delta_{p}{\phi_{1}}=\lambda_{1}{\phi_{1}}^{p-1}\;\mbox{in}\;\Om;\;\phi_1=0\;\mbox{on}\;\partial\Om.
\end{equation*}
Multiplying $\phi_{1}$ on both sides and then integrating we get,
\begin{equation*}
\int_{\Omega}{\vert\nabla{\phi_{1}}\vert}^{p}=\lambda_{1}\int_{\Omega}{\phi_{1}}^{p}.
\end{equation*}
If $u_n$ is the weak solution of the equation (\ref{dop}) then by Strong Maximum Principle \cite{Vaz} we have $\frac{\phi_1^p}{u_n^{p-1}}\in W^{1,p}_0(\Om)$ and hence using Picone Identity (Theorem 2.1 \cite{Ba}) we have,
\begin{eqnarray*}
&\int_{\Omega} |\nabla\phi_1|^p dx- \int_{\Omega}\nabla(\frac{\phi_1^p}{u_n^{p-1}})|\nabla u_n|^{p-2}\nabla u_n dx\geq{0}\\
&\mbox{or,}\;\;\int_{\Omega}({\lambda_{1}u^{p-1}-\lambda{f}_{n}{(u)}-u^{q}}){\phi_{1}^{p}}\geq{0}.
\end{eqnarray*}
where $f_n(u)=(u+\frac{1}{n})^{-\delta}$. 

Choose $\overline{\Lambda}=\max\limits_{x\in\Om}\cfrac{\lambda_{1}u^{p-1}-u^{q}}{f_{1}{(u)}}$\\
Now for every $\epsilon >{0}$ there exist a $\delta_{0}>{0}$ such that $s^q<\epsilon{s}^{p-1}$ for all $s\in [0,\delta_0]$. So for a suitable choice of $\epsilon$ we have $\Lambda>0$.\\  Let $\lambda\geq{\overline{\Lambda}}$ then we have,
\begin{equation}{\nonumber}
{\lambda}\geq\max_{\Om}\frac{\lambda_{1}u^{p-1}-u^{q}}{f_{1}{(u)}}>\frac{\lambda_{1}u^{p-1}-u^{q}}{f_{n}{(u)}}
\end{equation} 
which gives ${(\lambda_{1}u^{p-1}-\lambda{f}_{n}{(u)}-u^{q})}<0$. This is a contradiction.
\end{proof}

\begin{lem}\label{aprio}
Assume $\Om$ to be strictly convex. There exists $K>0$ independent of $n$ such that $||u_n||_{\infty}\leq K$ where $u_n$ solves (\ref{dop}).
\end{lem}

\begin{proof}[\textbf{Proof of Lemma \ref{aprio}}]
We will prove the lemma in several steps:\\

 \textbf{Step 1} (Uniform Hopf Lemma) We start by showing that for any $n\in \mathbb{N}$ we have $\frac{\pa u_n}{\pa\eta}(x)<c<0$ for some constant $c$ which is independent of $n$ but depends on $x$ and $\eta$ is the outward unit normal to $\pa\Om$ at the point $x$. \\
Since $\Om$ has a $C^2$ boundary it also satisfies the interior ball condition.  Hence for $x_0\in\pa\Om$ there exists $B_r(y)\subset\Om$ such that $\partial B_r(y)\cap\pa\Om=\{x_0\}$.\\
Define the function $w:B_r(y)\rtar \mathbb{R}$ such that
\begin{equation}\nonumber
w(x)=[2^{\frac{N-p}{p-1}}-1]^{-1} r^{\frac{N-p}{p-1}}|x-y|^{\frac{p-N}{p-1}}-
[2^{\frac{N-p}{p-1}}-1]^{-1}
\end{equation}
Hence $w$ satisfies the following:
\begin{enumerate}
\item $w(x)\equiv 1$ on $\pa B_{\frac{r}{2}}(y)$ and $w(x)=0$ on $\pa B_r(y)$.
\item $0<w(x)<1$ if $x\in  B_r(y)\setminus B_{\frac{r}{2}}(y)$ with $|\nabla w(x)|>c>0$ for some positive constant $c$ depending on $x$ .
\end{enumerate}

Define, $\tau=\inf\{u_n(x)\vert x\in \pa B_{\frac{r}{2}}(y)\}$.\\
Clearly $\tau>0$ is independent of $n$ thanks to the
fact that $u_n>c_{B_{\frac{2r}{3}}(y)}>0$ and since $\pa B_{\frac{r}{2}}\subset B_{\frac{2r}{3}}(y)$ by Lemma \ref{pursing}.
\\
Set $v=\tau w$ and note that $v$ satisfies the following equation:
\begin{eqnarray*}
&-\delp v=0\;\;\mbox{in}\;\; B_r(y)-\overline{B_{\frac{r}{2}}(y)}\\
&v=\tau\;\;\mbox{if}\;\; x\in \partial B_{\frac{r}{2}}(y);\;\; v=0\quad \mbox{if}
\;\; x\in\pa B_r(y)\nonumber
\end{eqnarray*}
 We also have that $u_n\geq v$ on the boundary of $B_r(y)-\overline{B_{\frac{r}{2}}(y)}$ and $\delp v\leq \delp u_n$ in $\Om$.\\
So the Weak Comparison Principle  implies $u_n\geq v$ in $B_r(y)-\overline{B_{\frac{r}{2}}(y)}$.\\
Now since $u_n(x_0)=v(x_0)=0$ so one has from properties of $w$:
\begin{equation*}
\begin{split}
\frac{\pa u_n}{\pa\eta}(x_0)&=\lim_{t\rtar 0}\frac{u_n(x_0-t\eta)}{t}\leq \lim_{t\rtar 0}\frac{v(x_0-t\eta)}{t}\\
&=\frac{\pa v(x_0)}{\pa\eta}=
\tau \frac{\pa w}{\pa\eta}<-c<0. 
\end{split}
\end{equation*}
where $c>0$ is independent of $n$.\\
\textbf{Step 2} (Existence of a neighbourhood of the boundary which is independent of critical points of $u_n$) Define, $Z(u_n)=\{x\in\Om: \nabla u_n(x)=0\}$ to be the the critical set of $u_n$.  Since $u_n\in C^1(\bar{\Om})$ from Step 1 we have that $\frac{\pa u_n}{\pa\eta}<0$ on the boundary. So using the compactness of $\pa\Om$ and $Z(u_n)$ we deduce that $\text{dist}(\pa\Om, Z(u_n))=d_n>0$ for all $n\in\mathbb{N}$. \\We assert that there exist $\epsilon_0>0$ independent of $n$ such that $d_n>\epsilon_0>0$ i.e, there exists a neighbourhood of boundary given by $\Om_{\epsilon_0}= \{x\in\Om:\text{dist}(x,\pa\Om)<\epsilon_0\}$ such that $Z(u_n)\cap \Om_{\epsilon_0}=\phi$. If not, then $\exists\; x_n\in Z(u_n)$ s.t $\text{dist}(x_n,\pa\Om)\to 0$ as $n\to\infty$. Upto a subsequence, $x_{n_k}\to y_0$. Clearly $y_0\in\pa\Om$ and let $\eta(y_0)$ is the unit outward normal to $y_0$ be such that $\frac{\pa u_n}{\pa\eta}(y_0)<c<0$, thanks to the Uniform Hopf Lemma.  Hence there exists $\iota>0$ such that for $x\in B_{\iota}(y_0)\cap \Om$ one has $|\nabla u_n(x)|>\frac{c}{2}$, where $c$ is independent of $n$. This is a contradiction since we can always choose $x_{n_0}\in B_{\iota}(y_0)\cap \Om$ such that $\nabla u_{n_0}(x_{n_0})=0$.\\
\textbf{Step 3} (Monotonicity of $u_n$) For $e\in\mathbb{S}^{n}$, $\ga\in\mathbb{R}$ and a fixed $n\in\mathbb{N}$ define
\begin{itemize}
\item The hyperplane $\mathbb{T}:=\mathbb{T}_{\ga,e}=\{x\in\mathbb{R}^N : x.e=\ga\}$ and the corresponding cap $\Sig=\Sig_{\ga,e}=\{x\in\mathbb{R}^N : x.e<\ga\}$.
\item $a(e)=\inf\limits_{x\in\Om} x.e$
\item  $x'=x_{\ga,e}$ be the reflection of $x$ w.r.t $\mathbb{T}$ i.e, $x'=x+2(\ga-x.e)e$.
\item ${\Sig}'$ be the non-empty reflected cap of $\Sig$ w.r.t $\mathbb{T}$ for any $\ga>a(e)$. 
\item $\Lambda_1(e):=\{\mu>a(e): \forall \gamma\in(a(e),\mu),\;\text{we have}\;(\ref{notouch})\;\text{holds}\}$ and $\Lambda'(e):=\sup \Lambda_1(e)$
\end{itemize} 
where (\ref{notouch}) is given by the following two condition:
\begin{itemize}\label{notouch}
\item  ${\Sig}'$ is not internally tangent to $\pa\Om$ at some point $p\notin T_{\gamma,e}$.
\item For all $x\in\pa\Om\cap T_{\gamma,e}$, $e(x).e\neq 0$ where $e(x)$ is the unit inward normal to $\pa\Om$ at $x$.
\end{itemize}
From Proposition 2 of Azizieh-Lumaire \cite{AzLe} we have that the map $e\to\Lambda'(e)$ is continuous, provided $\Om$ is strictly convex.\\
Further define, $v_n(x)=u_n(x_{\ga,e})$. Using the boundedness and the strict convexity of $\Om$ we have ${\Sig}'$ is contained in $\Om$ for any $\ga\leq \ga_1$, where $\ga_1$ depends only on $\Om$, independent of $e$. Define $\ga_0=\min(\ga_1,\epsilon_0)$. For $\ga-a(e)$ small consider any such $\Sigma$. Now since $v_n$ and $u_n$ both satisfies equation (\ref{dop}) and $\delp$ is invariant under reflection hence on the hyperplane $\mathbb{T}$ both functions coincides. Moreover for $x\in\pa\Sig\cap\pa\Om$ we have $u_n(x)=0$ and $v_n(x)=u_n(x')>0$ since $x'\in\Om$. 
Hence we have,
\begin{eqnarray*}
&\delp u_n+g_n(u_n)+f(u_n)=\delp v_n+g_n(v_n)+f(v_n)\;\;\mbox{in}\;\;\Sig\\
&u_n\leq v_n\;\;\mbox{on}\;\;\pa\Sig
\end{eqnarray*}
where $f_{n}(u)=(u+\frac{1}{n})^{-\delta}$ and $g(u)=u^{q}$.\\
Using the Comparison Principle of Damascelli-Scuinzi \cite{DaSc} for narrow domain we have $u_n\leq v_n$ in $\Sig$. Again using the Comparison Principle we have $u_n\leq v_n$ in $\Sig_{\ga,e}$ for any $\ga\in (a(e),\ga_0]$.\\
So $u_n$ is non-decreasing in the e-direction for all $x\in\Sig_{\ga_0,e}$.\\

\textbf{Step 4} (Existence of a non-zero measurable set away from boundary where $u$ is non-decreasing)\\
Fix $x_0\in\pa\Om$ and let $e=\eta(x_0)$ be the unit outward normal to $\pa\Om$ at $x_0$. From Step 3 we have that $u_n$ is non-decreasing in $e$ direction for all $x\in \Sig_{\ga,e}$ and $a(e)<\ga<\ga_0$.

If $\theta\in \mathbb{S}^{N-1}$ be any other direction close to $e$ then the reflection of $\Sig_{\gamma,\theta}$ w.r.t $\mathbb{T}_{\gamma,\theta}$ will still be in $\Om$ due to the strict convexity of the domain and so $u_n$ will be non-decreasing in the $\theta$ direction. Choose $\gamma=\frac{\gamma_0}{2}$ and consider the region $\Sig_{\frac{\gamma_0}{2},e}$, since $\Om$ is strictly convex there exists a small neighbourhood $\Theta\in \mathbb{S}^{N-1}$ such that $\Sig_{\frac{\gamma_0}{2},e}\subset \Sig_{\gamma_0,\theta}$ for all $\theta\in\Theta$.
Hence $u_n$ is non-decreasing in every direction $\theta\in\Theta$ and for any $x$ with $x.e<\frac{\ga_0}{2}$.\\
Set, $$\Sig_0=\{x\in\Om : \frac{\gamma_0}{8}<x.e<\frac{3\gamma_0}{8}\}$$ 
Clearly $\Sig_0\subset \Sig_{\frac{\gamma_0}{2},e} $ and $u_n$ is non-decreasing in any direction $\theta\in\Theta$ 
and $x\in\Sig_0$. Finally choose $\eps=\frac{\gamma_0}{8}$ and fix any point $x\in\Om_{{\eps}}$.
If $x_0$ is the projection of this point on $\pa\Om$ then 
\begin{equation*}
 u_n(x)\leq u_n(x_0-\eps e)\leq u_n(y)
\end{equation*}
for all $y\in I_x$ where $I_x\subset\Sig_0$ is the truncated cone with vertex at $x_0-{\eps}e$ and opening angle $\frac{\Theta}{2}$. Moreover $I_x$ has the following properties:
\begin{itemize}
\item $|I_x|>\kappa$ for some $\kappa$ depending only on $\Om$ and $\eps$.
\item $u_n(x)\leq u_n(y)$ for all $y\in I_x$ and $n\in\mathbb{N}$. 
\end{itemize}  

\textbf{Step 5} (Deriving the Boundary Estimates)
Using Picone's Identity (\cite{AlHu} or \cite{Ba}) on $e_1$ the first eigenvalue of the p-Laplacian on $\Om$ and
$u_n$ one has using the Strong Maximum Principle of Vazquez \cite{Vaz} that  $\frac{e_1^p}{u_n^{p-1}}\in W^{1,p}_0(\Om)$. \\
Therefore,
\begin{equation}\label{est}
\begin{split}
\int\limits_{\Om} \frac{[g_n(u_n)+f(u_n)]e_1^p}{u_n^{p-1}}&=\int\limits_{\Om} |\nabla u_n|^{p-2} \nabla u_n.\nabla(\frac{e_1^p}{u_n^{p-1}})\\
&\leq \int\limits_{\Om} |\nabla e_1|^p dx\leq C(\Om)
\end{split}
\end{equation}
Let $e_1(z)\geq \zeta>0$ for all $z\in\Om-\Om_{\frac{{\eps}'}{2}}$.
Hence from (\ref{est}) we deduce
\begin{equation*}
\zeta^p\int\limits_{\Om-\Om_{\frac{{\eps}}{2}}}\frac{[g_n(u_n)+f(u_n)]}{u_n^{p-1}}\leq C(\Om).
\end{equation*}
which would then imply that 
 \begin{equation*}
\int\limits_{I_x}\frac{[g_n(u_n)+f(u_n)]}{u_n^{p-1}}\leq \frac{C(\Om)}{\zeta^p}.
\end{equation*}
Now since, 
\begin{equation}
\int\limits_{I_x}\frac{[g_n(u_n)+f(u_n)]}{u_n^{p-1}}\geq \int\limits_{I_x} u_n^{q-p+1}(y) dy\geq u_n^{q-p+1}(x)|I_x|
\end{equation}
we have, 
\begin{equation*}
u_n^{q-p+1}(x)\leq \frac{C'(\Om)}{\zeta^p}\;\text{for some constant} \;C'>0
\end{equation*}
i.e, $u_n(x)\leq \bar{C}$ for all $x\in\Om_{\eps}$ and for all $n\in\mathbb{N}$.\\

\textbf{Step 6} (Initiating the Blow-up Analysis) For any open set ${\Om}'\subset\subset\Om$ there exists $C({\Om}')$ such that $||u||_{\infty}<C({\Om}')$ for every solution $u_n$ of $(P_{n,\lambda})$. \\
Assume by contradiction that there is a sequence $(u_n)$ of positive solutions of $(P_{n,\lambda})$ and a sequence of points $x_n\in\Om$ such that  $M_n=u_n(P_n)=\max\{u_n(x): x\in\bar{{\Om}'}\}\to\infty $ as $n\to\infty$.
Using the boundary estimates we can safely assume that $x_n\to x_0\in \bar{{\Om}'}$ as $n\to\infty$.
Let $2d$ be the distance of $\bar{{\Om}'}$ to $\pa\Om$ and assume $\Om_d=\{x\in\Om : \text{dist}(x,{\Om}')<d\}$\\
Let $R_n$ be the sequence of positive numbers such that $R_n^{\frac{p}{q-p+1}}M_n=1$. Clearly $R_n\to 0$ as $M_n\to\infty$.\\
Define the scaled function $v_n:B_{\frac{d}{R_n}}(0)\to \mathbb{R}$ such that
\begin{equation}\nonumber
v_n(y)=R_n^{\frac{p}{q-p+1}} u_n(P_n+R_n y)
\end{equation} 
Since $u_n$ attains its maxima at $P_n$ we have $||v_n||_{\infty}=v_n(0)=1$.\\
Again $R_n\to 0$ we can choose a $n_0$ such that $B_R(0)\subset B_{\frac{d}{R_n}}(0)$ for a fixed $R>0$.\\
Also we have that $v_n$ satisfies the following:
\begin{equation*}
\begin{split}
&\nabla v_n(y)=R_n^{\frac{p}{q-p+1}+1}\nabla u_n(P_n+R_n y)\\
\mbox{and},\;\;&\delp v_n(y)=R_n^{\frac{pq}{q-p+1}} [\lam f_n(u_n(P_n+R_n y))+R_n^{\frac{-pq}{q-p+1}} v_n^q(P_n+R_n y)]
\end{split}
\end{equation*}
Since $P_n+R_n y\in \bar{\Om}_d\subset\Om$ for any $y\in B(0,R)$ we have from Lemma \ref{pursing} and Theorem \ref{scp2}, \[R_n^{\frac{pq}{q-p+1}} [\lam f_n(u_n(P_n+R_n y))+R_n^{\frac{-pq}{q-p+1}} v_n^q(P_n+R_n y)]\leq C(\bar{\Om}_d)\] for all $n\geq n_0$.  Fixing a ball $\bar{B}\in B(0,\frac{d}{R_n})$ for all $n\geq n_0$, from the Interior estimates of Tolksdorf \cite{To} and Lieberman \cite{Lie} we get the existence of some constant $K>0$ and $\beta\in(0,1)$ depending only on $N,p,B$ such that \[v_n\in C^{1,\beta}(\bar{B})\;\;\mbox{and}\;||v_n||_{1,\beta}\leq K\] This allows us to deduce the existence of a function $v\in C^1(\bar{B})$ and a convergence subsequence $v_n\to v$ in $C^1(\bar{B})$ from Ascoli-Arzela theorem. Passing to the limit we have, 
\begin{eqnarray*}
&\int_{B} |\nabla v|^{p-2}\nabla v\cdot\nabla\phi\geq C\int_B v^q\phi;\;\;\phi\in C_c^{\infty}(B)\\
&v\in C^1(\bar{B}),\;\;v\geq 0 \;\mbox{on}\;\bar{B}
\end{eqnarray*}
Moreover we also have, $||v||_{\infty}=1$. Using Strong Maximum Principle of Vazquez \cite{Vaz}
we also have, $v(x)>0$ for all $x\in B$. Taking larger and larger balls we obtain a Cantor diagonal
subsequence which converges to $v\in C^1(\mathbb{R}^N)$ on all compact subsets of $\mathbb{R}^N$ and satisfy
\begin{eqnarray*}
&\int_{\mathbb{R}^N} |\nabla v|^{p-2}\nabla v\cdot\nabla\phi\geq C\int_{\mathbb{R}^N} v^q\phi;\;\;\phi\in C_c^{\infty}(\mathbb{R}^N)\\
&v\in C^1(\mathbb{R}^N),\;\;v> 0 \;\mbox{in}\;\mathbb{R}^N
\end{eqnarray*} 
which is a contradiction to the Liouville theorem of Mitidieri-Pohozaev (Theorem \ref{liou}).
\end{proof}

\begin{lem}\label{exis}
Assume $\Om$ to be strictly convex. Then there exists $\Lambda>0$ such that for $0<\lambda<\Lambda$ and any $\delta>0$ the problem (\ref{dop}) admits atleast two solution $u,v\in W^{1,p}_{loc}(\Omega)$. Moreover there exists $\alpha>0$ such that $u^{\alpha}, v^{\alpha}\in W^{1,p}_0(\Omega)$. 
\end{lem}

Before we begin with the proof of Lemma \ref{exis} we state some lemmas. We will provide proof in cases where they are generalized for p-laplacian.
  
\begin{lem} [DeFigueiredo et al \cite{DeLiNu}]\label{deg}
Let $C$ be a cone in a Banach space $X$ and $\phi:{C}\rightarrow{C}$ be a compact map such that $\phi(0)=0$. Assume that there exists $0<r<R$ such that 
\begin{enumerate}[(1)]
\item $x\neq{t\phi{(x)}}$ for $0\leq t\leq 1$ and $||x||=r$
\item a compact homotopy $F:\overline{B_{R}}\times[0,\infty)\rightarrow{C}$ such that $F(x,0)=\phi(x)$ for $\vert\vert{x}\vert\vert=R,\;F(x,t)\neq{x}\;for\;\vert\vert{x}\vert\vert=R$ and $0\leq{t}<\infty$ and $F(x,t)=x$ has no solution for $x\in\overline{B_{R}}$ for $t\geq{t_{0}}.$
\end{enumerate}
Then if, $U=\{x\in{C}:r<\vert\vert{x}\vert\vert<R\}$ and $B_{\rho}=\{x\in{C}:\vert\vert{x}\vert\vert<\rho\}$ we have $deg(I-\phi,B_{R},0)=0,deg(I-\phi,B_{r},0)=1$ and $deg(I-\phi,U,0)=-1$
\end{lem}

Let us define the set 
$$\mathbb{P}=\{u\in C^{1,\alpha}_0(\bar\Omega): u(x)\geq 0\;\;\text{in}\;\;\overline{\Omega}\}$$ 
Clearly $$\mathbb{P}^{\sim}=\{u\in C^{1,\alpha}_0(\bar\Omega): u(x)> 0\;\mbox{and}\;\frac{\partial u}{\partial \eta}(x)<0\; \text{for all}\; x\in\partial\Om\}$$ is the interior of $\mathbb{P}$, where $\eta$ is the unit outward normal to $\partial\Om$.

\begin{lem}\label{gam}
Assume that $\overline{u}$, $u$ belong to $C_{0}^{1,\alpha}(\bar{\Omega})$ satisfying the equation
$$
-\Delta_{p}\overline{u}>\lambda f_{n}(\overline{u})+g(u)\quad\text{in}\;\;\Omega
$$ and
$$
-\Delta_{p}u=\lambda f_{n}(u)+g(u)\quad\text{in}\;\;\Omega
$$ respectively. If $u\neq \overline{u}$, then we have $\overline{u}-u$ does not belong to $\partial\mathbb{P}$.
\end{lem}

\begin{proof}
Assume $\overline{u}-u\in\partial\mathbb{P}$. Hence we have, $\overline{u}(x)\geq u(x)$.  Using Theorem \ref{scp2} we get $\overline{u}-u\in\mathbb{P^\sim}$. Since $\mathbb{P^\sim}\cap\partial\mathbb{P}=\emptyset$, we arrive at a contradiction to our assumption.
\end{proof}
\begin{lem}\label{fro}
Suppose $I\subset\mathbb{R}$ is an interval and let $\sum\subset{I}\times{C}_{0}^{1,\alpha}(\overline{\Omega})$ be a connected set of solutions of equation (\ref{dop}).
Consider a continuous map $U: I\rightarrow {C}_{0}^{1,\alpha}(\overline{\Omega})$ such that $U(\lambda)$ satisfies 
$$
-\Delta_{p}U(\lambda)>\lambda f_{n}(u)+g(u_n)\quad\text{in}\;\;\Omega\;\;\forall\;\;\lambda\in I.
$$
If $u_{0}\leq U(\lambda_{0})$ in $\Omega$, $u_0\neq U(\lambda_{0})$ for some $(\lambda_{0},u_{0})\in{\sum}$ then $u<U(\lambda)$ in $\Omega$ for all $(\lambda,u)\in{\sum}$.  
\end{lem}

\begin{proof}
Consider a continuous map,
\begin{equation}{\nonumber}
T:I\times{C_{0}^{1,\alpha}(\overline{\Omega})}\rightarrow{C_{0}^{1,\alpha}(\overline{\Omega})}\;\mbox{given by}\;T(\lambda,u)=U(\lambda)-u.
\end{equation}
Since $T$ is a continuous operator, $T(\sum)$ is connected in ${C_{0}^{1,\alpha}(\overline{\Omega})}.$\\
By Lemma \ref{gam}, $T(\sum)$ completely lies in $P^{\sim}$ or completely outside $P.$ 
Since $T(\lambda_0,u_0)\in\mathbb{P}$, we have $T(\sum)\subset\mathbb{P^\sim}$ and therefore $u<U(\lambda)$ for all $(\lambda,u)\in{\sum}$.
\end{proof}

\begin{lem}(Ambrosetti-Arcoya \cite{AmAr})\label{continum}
Given $X$ be a real Banach space with $U\subset X$ be open, bounded set. Let $a,b\in\mathbb{R}$ such that the equation $u-T(\lambda,u)=0$ has no solution on $\partial{U}$ for all $\lambda\in [a,b]$ and that $u-T(\lambda,u)=0$ has no solution in $\overline{U}$ for $\lambda={b}$.\\
Also let $U_{1}\subset{U}$ be open such that $u-T(\lambda,u)=0$ has no solution in $\partial{U_{1}}$ for $\lambda={a}$ and $\mbox{deg}(I-K_{a},U_1,0)\neq{0}$. \\Then there exists a continuum $C$ in $\sum=\{(\lambda,u)\in[a,b]\times{X}:u-T(\lambda,u)=0\}$
such that 
\begin{equation*}
{C}\cap(\{a\}\times{U}_{1})\neq{\emptyset}\;\;\mbox{and}\;\; {C}\cap(\{a\}\times(U-U_{1}))\neq{\emptyset}
\end{equation*}
\end{lem}

\begin{proof}[Proof of Lemma \ref{exis}]
\textbf{Step 1:} Define, $A(s)=\frac{1}{2}\Big((\frac{s}{T})^{\delta+p-1} -s^{\delta +q}\Big)$ for $s\in [0,\infty)$ and $T$ is as in Lemma \ref{aprio} and define $$\beta=\max_{0\leq s\leq \delta_2} A(s)$$ where $\delta_1=\frac{1}{2}(2q-2p+3)^{\frac{1}{p-q-1}}T^{\frac{\delta+p-1}{p-q-1}}$.
Clearly for $\delta_2\in(0,\min\{\delta_0,\delta_1\})$ we have, $A$ is strictly positive on $(0,\delta_2)$ and so $\beta>0$. Hence by I.V.P of continuous functions there exists a $\mu\in (0,\delta_2)$ such that $A(\mu)=\lambda_0$.\\
If we set $\lambda_{*}=(\frac{\mu}{T})^{\delta+p-1}$ then $$\lambda_{*}>\lambda_0+{\mu}^{\delta+q}=\lambda_0+[T (\lambda_{*})^{\frac{1}{\delta+p-1}}]^{\delta+q}$$
Hence for $w_{n,\lambda_{*}}$ satisfying equation (\ref{pursinre}) and $n\geq n_0$ one has, \[\lambda_{*}>\lambda_0+ ||w_{n,\lambda_{*}}||^q_{\infty}\big(||w_{n,\lambda_{*}}||_{\infty}+\frac{1}{n}\big)^{\delta}\] which can be rewritten as \[\lambda_{*}> \lambda +w_{n,\lambda_{*}}^q\big(w_{n,\lambda_{*}}+\frac{1}{n}\big)^{\delta}\;\;\mbox{for}\;\;\lambda\leq \lambda_0\]
Therefore \[-\delp w_{n,\lambda_{*}}=\frac{\lambda_{*}}{\big(w_{n,\lambda_{*}}+\frac{1}{n}\big)^{\delta}}> \frac{\lambda}{(w_{n,\lambda_{*}}+\frac{1}{n})^{\delta}} +w_{n,\lambda_{*}}^q\;\;\mbox{for}\;\;\lambda\leq \lambda_0\;\;\mbox{ and}\;\; n\geq n_0\]
Hence we have the existence of a super solution $w_{n,\lambda_{*}}\in C^{1,\alpha}(\bar\Omega)$ for some $\alpha>0$ with $||w_{n,\lambda_{*}}||_{\infty}\leq \mu$ which is not a solution to (\ref{dop}).\\

\textbf{Step 2:} Define $$F_n(s)=\frac{\lambda (s+\frac{1}{n})^{-\delta} +s^q}{s^{p-1}}$$ for $s\in [0,\infty)$. Using the convexity of the function $s^q(s+\frac{1}{n})^{1+\delta}$ we can derive the existence of a unique $M_n>0$ which is increasing w.r.t $\lambda$ such that \[\lambda(p+\delta-1)M_n+\frac{p-1}{n}=(q-p+1) M_n^q(M_n+\frac{1}{n})^{1+\delta}\] Moreover one also have, \[(q-p+1) s^q(s+\frac{1}{n})^{1+\delta}\leq \lambda(p+\delta+1)s+\frac{p-1}{n}\] for $s\leq M_n$. From the above we can conclude that $$F'_n(s)=\frac{1}{s^p}\Big[\frac{\lambda(1-p-\delta)s+\frac{1-p}{n}}{(s+\frac{1}{n})^{1+\delta}}\Big]+(q-p+1) s^{q-p}<0$$ \\Hence $F_n$ is decreasing and by Diaz-Saa \cite{DiSa} we have the existence of a unique solution to equation (\ref{dop}) s.t $||u_n||_{\infty}\leq M_n$. \\Again from Step 1 we have for $\mu<\delta_1$, \[\frac{q-p+1}{\delta+p-1}{\mu}^{\delta+q}<\lambda_0\] provided $\delta>1$. So \[M_n(\lambda_0)\geq M_n(\lambda_n)=\mu+\epsilon\] for all $n\geq m_1$ where $\lambda_m$ is defined as $$\lambda_m:=\frac{(q-p+1)(\mu+\epsilon)^q(\mu+\epsilon+\frac{1}{m})^{1+\delta}}{(\mu+\epsilon)(\delta+p-1)+\frac{p-1}{m}}<\lambda_0$$

Step 3: (Existence of atleast two solution) Write $f_{n}(u)=(u+\frac{1}{n})^{-\delta}$ and $g(u)=u^{q}$ for $n\geq \max(n_0,m_1)$.\\
Define $K_{\lambda}:C(\bar{\Om})\rightarrow C(\bar{\Omega})$ by 
$$K_{\lambda}{(u)}={(-\Delta_{p})^{-1}}(\lambda{f}_{n}{(u)}+g(u));\;\;\lambda\geq 0$$
Using the compact of $(-\delp)^{-1}$ on $C(\bar\Om)$ we can assume that $K_{\lambda}$ is also compact map. Note that one can view equation (\ref{dop}) as the fixed point equation given by  
$u=K_{\lambda}(u)$.\\
Recall from Lemma \ref{nonex} we have equation (\ref{dop}) does not admit any solution for $\lambda\geq \bar\Lambda$. So for $\lambda\in [0,\bar{\Lambda}]$, from Lemma \ref{aprio} we have, ${\vert\vert{u_{n}}\vert\vert}_{\infty}\leq{R}$.\\
Consider the positive cone of $X:=C(\bar\Om)$ given by:
\begin{equation}{\nonumber}
C=\{u\in C(\overline{\Omega}):u\geq{0}\;\mbox{in}\;\Omega\}
\end{equation} 
Define
\begin{equation*}
K_0:C\rightarrow{C}\;\mbox{by}\;\;K_0(u)=(-\Delta_{p})^{-1} g(u)
\end{equation*}
and,
\begin{equation*}
F:\bar{B}_{R}\times[0,\infty)\rightarrow{C}\;\mbox{by}\;
F(u,\lambda)={(-\Delta_{p})^{-1}}(\lambda{f}_{n}{(u)}+g(u)).
\end{equation*}
Using Lemma \ref{wosing}, Lemma \ref{nonex} and Lemma \ref{aprio} we conclude that $K_0$ and $F$ satisfies all the conditions in Lemma \ref{deg} for some $0<r<\mu<R$.  Since $\mu<\delta_0$ so $(I-K_0)u$ has no solution on $\pa B_r$.\\
Hence we have, $\text{deg}(I-K_0,B_R,0)=0$ and $\text{deg}(I-K_0,B_r,0)=1$.\\
Keeping in mind Lemma \ref{nonex} and setting $a=0, b=\bar{\Lambda},T(\lambda,u)=K_{\lambda}(u), U=B_{R}\;\mbox{ and}\; U_1=B_r$  in Lemma \ref{continum}
we get a continuum $C_n\subset\sum=\{(\lambda,u)\in [0,\bar\Lambda]\times X : u-K_{\lambda}(u)=0\}$ such that
\begin{equation}\label{hyuo}
C_n\cap (\{0\}\times B_r)\neq \emptyset,\;C_n\cap(\{0\} \times (B_{R}-B_r))\neq\emptyset
\end{equation}

Define the continuous map $U:[0,\lambda_{0}]\rightarrow{C}_{0}^{1,\alpha}{(\overline{\Omega})}$ by $U(\lambda)=w_{n,\lambda^{*}}\;\forall\;\lambda\in[0,\lambda_{0}]$\\
By lemma \ref{fro} to deduce that every pair $(\lambda,u_{n})$ belonging to the connected component of ${C}_{n}\cap({[0,\lambda_{0}]}\times{C(\bar{\Omega})})$ which emanates from $(0,0)$ lies pointwise below the branch $\{(\lambda,U(\lambda)):{0}\leq\lambda\leq\lambda_{0}\}$ at least until it crosses $\lambda=\lambda_{0}$. \\
In particular there exists $u_n$ in the slice $C_n^{\lambda_0}=\{u\in{C(\overline{\Omega})}:(\lambda_{0},u)\in{C_{n}}\}$ which satisfies that $0<u_n<w_{n,\lambda^{*}}$.
Recalling that $\vert\vert{w_{n,\lambda^{*}}}\vert\vert\leq{\mu}$ we
have ${\vert\vert{u_{n}}\vert\vert}_{\infty}\leq{\vert\vert{w_{n,\lambda^{*}}}\vert\vert}_{\infty}\leq{\mu}$\\
Clearly, from Step 2 we have, $u_n$ is the unique solution of equation ($\ref{dop}$) with small norm e.g,
$\vert\vert{u_n}\vert\vert_{\infty}\leq{\mu+\epsilon}$.\\
 Again by (\ref{hyuo}) one has, $C_n\cap(\{0\} \times (B_{R}-B_{\mu+\epsilon}))\neq\emptyset $ and so we conclude also the existence of $v_n$ such that $\vert\vert{v_{n}}\vert\vert_{\infty}\geq{\mu+\epsilon}$.\\
 Hence we have the existence of two distinct solution for $\lambda=\lambda_0$, since $\lambda_0< \bar\Lambda$ is arbitrary we have our required result.
\end{proof}

\begin{proof}[Proof of Theorem \ref{MT}]
From Lemma \ref{exis} we have the existence of atleast two solution $u_n$ and $v_n$ solving equation (\ref{dop}). \\
Note that we can choose $c>0$ such that $\underline{u}=(c\phi_{1}+n^{\frac{1+p-\delta}{p}})^{\frac{p}{\delta+p-1}}-\frac{1}{n}$ will be a weak sub-solution to the problem (\ref{pursinre}) for $\lambda=\lambda_0$.\\
Since $\frac{\lambda_{0}}{(s+\frac{1}{n})^{\delta}}\leq{\frac{\lambda_{0}}{(s+\frac{1}{n})^{\delta}}}+s^{q}$ for $s\geq{0}$ so one concludes that each solution of (\ref{dop}) with $\lambda=\lambda_{0}$ is a super-solution of (\ref{pursinre}) with $\lambda=\lambda_{0}$\\
Using the Strong Comparison Principle (Theorem \ref{scp2}) we have
\begin{equation}\label{unifo}
\underline{u}\leq{w_{n,\lambda_{0}}}\leq{u_{n}}\leq\mu,\;\underline{u}\leq{w_{n,\lambda_{0}}}\leq{v_{n}}\;\mbox{and}\;\vert\vert{v_{n}}\vert\vert_{\infty}\geq{\mu+\epsilon}>\mu
\end{equation}
Let $z_{n}=u_{n}$ or $v_{n}$ so from (\ref{unifo}) and Lemma \ref{aprio} we have, $$\underline{u}\leq z_n\leq M$$ where $M$ is independent of $n$. 
By Lemma \ref{pursing} and Strong Comparison Principle (Theorem \ref{scp2})
\begin{equation}\label{you}
\forall\;{\omega\subset\subset{\Omega}},\;\exists\;{c_{\omega}}:z_{n}\geq{c_{\omega}}>{0}\;\;\mbox{in}\;\;\omega\;\;\mbox{and for all}\;\;n\;\in\mathbb{N}
\end{equation}
We now claim that $z_{n}$ is bounded in $W_{loc}^{1,p}{(\Omega)}.$\\
Let $\phi\in{C_{0}^{1}{(\Omega)}}$ and take $z_{n}\phi^{p}$ as test function in equation (\ref{dop}) we get
\begin{equation}{\nonumber}
\int_{\Omega}{\vert\nabla{z_{n}}\vert}^{p}\phi^{p}=-p\int_{\Omega}\phi^{p-1}z_{n}{\vert\nabla{z_{n}}\vert}^{p-2}\nabla{\phi}.\nabla{z_{n}}+
\int_{\Omega}\frac{\lambda_{0}z_{n}\phi^{p}}{(z_{n}+\frac{1}{n})^{\delta}}+\int_{\Omega}z_{n}^{q+1}\phi^{p}.
\end{equation}
Again using Young's Inequality with $\epsilon$ we have, 
$\int_{\Omega}{\vert\nabla{z_{n}}\vert}^{p}\phi^{p}\leq{c_{\phi}}\;\mbox{for all}\;n\in{\mathbb{N}}$, where $c_{\phi}>0$
is a constant depending on $\phi$.\\ 
So $z_{n}\in{W_{loc}^{1,p}(\Omega)}$.\\
Hence there exists ${z}\in{W}_{loc}^{1,p}{(\Omega)}\cap{L^{\infty}}({\Omega})$ such that upto a subsequence $z_{n}\rightarrow{z}$ a.e to $z$ and weakly in ${W^{1,p}}{(\omega)}$ for all $\omega\subset\subset\Om$.\\
The convergence of $\int_{\Om} |\nabla u_n|^{p-2} \nabla u_n\nabla\phi\to\int_{\Om}|\nabla u|^{p-2} \nabla u\nabla\phi$ follows as in Theorem 4.4 of Canino et al \cite{CaScTr}. Again applying Dominated convergence theorem we deduce that,
\begin{equation*}
\lim_{n\to\infty}\int_{\Omega}\big(\frac{\lambda_{0}}{(z_{n}+\frac{1}{n})^{\delta}}+\phi{z}_{n}^{q}\big)dx=\lambda_{0}\int_{\Omega}\frac{\phi}{z^{\delta}}+\int_{\Omega}\phi{z}^{p}dx
\end{equation*}

Again since ${\vert\vert{u_{n}}\vert\vert}_{\infty}\leq{\mu}$, ${\vert\vert{v_{n}}\vert\vert}_{\infty}\geq\mu+\epsilon>\mu$ and $u_{n}\rightarrow{u}, v_{n}\rightarrow{v}$ a.e we have the existence of two distinct solution $u$ and $v$.

Now we will prove that for some $\alpha>0$ we have $u^{\alpha},v^{\alpha}\in{W}_{0}^{1,p}{(\Omega)}$.

Fix $\alpha>\frac{(p-1)(\delta+p-1)}{p^2}$ and $\theta=p(\alpha-1)+1$ hence, $\theta>\frac{(\delta-1)(p-1)}{p}$\\
Take $\phi=(z_{n}+\frac{1}{n})^{\theta}-(\frac{1}{n})^{\theta}$ as a test function in equation $(\ref{dop})$ to obtain
\begin{equation}\nonumber
\begin{split}
\int_{\Om} |\nabla\big((z_n+\frac{1}{n})^{\alpha}-\frac{1}{n^{\alpha}}\big)|^p dx &={\alpha}^p \int_{\Om}(z_n+\frac{1}{n})^{(\alpha-1)p} |\nabla z_n|^p dx\\
&\leq \lambda_0 \int_{\Om} (z_n+\frac{1}{n})^{\theta-\delta}+\int_{\Om} (z_n+\frac{1}{n})^{\theta} z_n^q dx\\
&\leq \lambda_0 \int_{\Om} (z_n+1)^{\theta-\delta}+\int_{\Om} (z_n+\frac{1}{n})^{\theta} z_n^q dx
\end{split}
\end{equation}
Now if $\theta\geq\delta$ then the above integration is bounded thanks to Lemma \ref{aprio}.\\
If $\theta<\delta$ then we have,
\begin{equation}{\nonumber}
\begin{aligned}
(z_{n}+\frac{1}{n})^{\theta-\delta}
& \leq{(c\phi_1+n^{\frac{\delta+p-1}{p}})^{\frac{p(\theta-\delta)}{\delta+p-1}}}\leq (c\phi_{1})^{\frac{p(\theta-\delta)}{\delta+p-1}}
\end{aligned}
\end{equation}
Since, $\theta>\frac{(\delta-1)(p-1)}{p}$ hence $\int_{\Om}\phi_{1}^{\frac{p(\theta-\delta)}{\delta+p-1}} dx<\infty$ (See Mohammed \cite{Mo})\\
Therefore $(z_{n}+\frac{1}{n})^{\alpha}-(\frac{1}{n})^{\alpha}$ is bounded in $W_{0}^{1,p}({\Omega})$ and since $z_n$ converges a.e to $z$ in $\Om$ we have, $(z_{n}+\frac{1}{n})^{\alpha}-(\frac{1}{n})^{\alpha}\rightarrow^{w} z^{\alpha}$ a.e in $W_{0}^{1,p}{(\Omega)}$.
\end{proof}

\section*{Useful Results}
\begin{thm}\label{stam} (Stampacchia \cite{Guido})
Assume $\phi:[0,\infty)\rightarrow[0,\infty)$ is a
non-increasing function such that if $h>k>k_0$ for some
$\alpha>0,\beta>1$ and $\phi(h)\leq\frac{d}{(h-k)^{\alpha}}[\phi(k_{0})]^{\beta-1}$ then $\phi(k_{0}+T)=0$ where $T^{\alpha}=d 2^{\frac{\alpha\beta}{\beta-1}}[\phi(k_{0})]^{\beta-1}$
\end{thm}

\begin{thm}(Strong Comparison Principle \cite{GuVe})\label{scp2}
Assume $p>1$ and $\Om$ be a bounded, connected open subset of $\mathbb{R}^N$ with $C^2$ boundary. If $f, g\in L^{\infty}(\Om)$ and $u,v\in C^1(\bar\Om)$ be two solution of $-\delp u=f\;\mbox{in}\;\Om;\;\;u|_{\partial\Om}=0$ and $-\delp v=g\;\mbox{in}\;\Om;\;\;v|_{\partial\Om}=0$. If we assume also that $f\geq g\geq 0$ a.e. in $\Omega$ and that the set $\mathbb{C}=\{x\in\Om: f(x)=g(x)\}$ 
has empty interior then $$u(x)>v(x)\geq 0\;\;\mbox{for}\;\;x\in\Om\;\;\mbox {and}\;\; \frac{\partial u}{\partial \eta}<\frac{\partial v}{\partial \eta}\leq 0\;\mbox{for all }\; x\in\partial\Om$$.
\end{thm}

\begin{thm}\label{liou}(Mitidieri-Pohozaev \cite{MiPo})
If $p-1<q<\frac{N(p-1)}{N-p}$, $p<N$ and $C>0$ then the problem \[\int_{\mathbb{R}^N} |\nabla w|^{p-2}\nabla w\cdot\nabla\phi\geq C\int_{\mathbb{R}^N} w^q\phi;\;\;\phi\in C_c^{\infty}(\mathbb{R}^N)\]
has no positive solution in $C^1(\mathbb{R}^N)$.
\end{thm}

\section*{Acknowledgement} The first author would like to thank Prof Jacques Giacomoni on some fruitful discussion of the topic. The first author is supported by DST-Inspire Faculty Award MA-2013029.
The second author is supported by NBHM Fellowship No: 2/39(2)/2014/NBHM/R$\&$D-II/8020/June 26,2014

\section*{References}

\bibliography{mybibfile}

\end{document}